\documentclass{article}[12pt]
\usepackage{amsmath}
\usepackage{amsthm}
\usepackage{amsfonts}
\usepackage{amssymb}
\usepackage{amstext}
\usepackage{amsopn}
\usepackage{latexsym}
\newtheorem{thm}{Theorem}

\newtheorem{cor}{Corollary}
\newtheorem{lemma}{Lemma}[section]

\begin{document}
\title{On Norms of Iterations of $\{0,1\}$-Matrices
\thanks{Supported by the NSFC (Nos. 11601161 and 11771153) and
the Fundamental Research Funds for the Central Universities: Zhongnan University of Economics and Law(31511911211).}}

\author{Chun WEI   \\
School of Statistics and Mathematics \\
Zhongnan University of Economics and Law \\ Wuhan 430073, China\\
E-mail: hbxt1986v@163.com\\ \\
Fan WEN \thanks{Corresponding author: wenfan@math.tsukuba.ac.jp}\\ College of Mathematics\\
  School of Science and Engineering
  \\University of Tsukuba
  \\ Ibaraki 305-8575, Japan\\
E-mail: wenfan@math.tsukuba.ac.jp}

\maketitle

\begin{abstract}
Let $M$ be a $b\times b$ nonzero $\{0,1\}$-matrix. Let $\rho(M)$ be its spectral radius and let $\|M^n\|$ be the norm of its $n$-th iteration. In the case $\rho(M)>1$, we see from the spectral radius formula that $\{\|M^n\|\}_{n=1}^\infty$
tends to $\infty$ exponentially as $n\to\infty$. In the case $\rho(M)=1$, $\{\|M^n\|\}_{n=1}^\infty$ can be bounded or tend to $\infty$ depending on $M$.
The fine behavior of this sequence is completely characterized in the present paper.

\medskip


\end{abstract}

\section{Notations and Main Results}

Let $\mathbb{N}$ be the set of positive integers. Let $k,m\in \mathbb{N}$. Let
$$M=(M_{ij})=\left(
  \begin{array}{cccc}
    M_{11} & M_{12} & \cdots & M_{1m} \\
    M_{11} & M_{22} & \cdots & M_{2m}\\
     &  & \cdots &  \\
    M_{k1} & M_{k2} & \cdots & M_{km} \\
  \end{array}
\right)$$
be a $k\times m$ nonzero matrix. Denote by $r_i(M)$ its $i$-th row and by $c_j(M)$ its $j$-th column.
Define the norm of $M$ by
\begin{equation}\label{nm}
\|M\|=\sum_{i=1}^{k}\sum_{j=1}^{m} |M_{ij}|.
\end{equation}
If $M$ is a square matrix, we call $$\rho(M):=\max\{|\lambda|: \mbox{$\lambda$ is an eigenvalue of $M$}\}$$
the spectral radius of $M$. Brualdi and Hwang \cite{BH} studied the spectral radius of $\{0,1\}$-matrices with $1$'s in prescribed positions.

It is well known that
$$
\rho(M)=\lim_{n\to\infty}\|M^n\|^{1/n};
\mbox{ see Proposition 4.4.1 in \cite{S02}.}
$$
Therefore, if $\rho(M)\neq 1$, the sequence $\{\|M^n\|\}_{n=1}^\infty$
tends to $0$ or $\infty$ exponentially. In this note we study the behavior of this sequence for $\{0,1\}$-matrices $M$ with $\rho(M)=1$. We completely solve this question.

\medskip

Denote by $\bf{1}$ any matrix of all elements equal to $1$ and by $\bf{0}$ any zero matrix. Their patterns are self-evident when they appear. Let $b\in \mathbb{N},\, b\geq 2$. Let $A=\{1,2,\cdots, b\}$. Let $\mathcal{M}_b$ denote the family of $b\times b$ nonzero $\{0,1\}$-matrices. We say that a matrix $M$ in $\mathcal{M}_b$ satisfies the condition ${\bf P_1}$, if
$$
\mbox{ $r_i(M)\neq\bf{0}$ for any $i\in\{j\in A: c_j(M)\neq \textbf{0}\}$.}
$$
This condition is stronger than that  $M$ is non-nilpotent. Let
\begin{equation}\label{HM}
\mathcal{M}_b({\bf P_1})=\{M\in \mathcal{M}_b: M\mbox{ satisfies the condition ${\bf P_1}$}\}.
\end{equation}

\medskip

The main results of this paper are the following three theorems.

\begin{thm}\label{t1}
Let $M\in\mathcal{M}_b({\bf P_1})$. Then $\{\|M^n\|\}_{n=1}^\infty$ obeys the trichotomy:

(1) $\{\|M^n\|\}_{n=1}^\infty$ tends to $\infty$ exponentially.

(2)  $n+2\leq\|M^n\|\leq C_{n+b}^{n+1}$ for all $n\in\mathbb{N}$.

(3) $\|M^n\|\leq 2^{b-1}$ for all $n\in\mathbb{N}$.

\noindent Hereafter $C_k^m$ denotes the number of combinations.
\end{thm}

Let $k\in\mathbb{N}$. Let $L_k=(L_{ij})$ and $T_k=(T_{ij})$ be two $k\times k$ matrices defined respectively by
$$
L_{ij}=\left\{
         \begin{array}{ll}
           1 & \hbox{if $i>j$} \\
           0 & \hbox{otherwise}
          \end{array}
       \right.
$$
and
$$
T_{ij}=\left\{
         \begin{array}{ll}
           1 & \hbox{if $i\geq j$} \\
           0 & \hbox{otherwise.}
          \end{array}
       \right.
$$
Then $L_k^k={\bf 0}$ by a simple computation. Let $I_k$  be the $k\times k$ unit matrix. For $k\geq 2$ denote by $J_k=(J_{ij})$ the $k\times k$ matrix whose elements are all zero except
$$J_{12}=J_{23}=\cdots=J_{(k-1)k}=J_{k1}=1.$$
It is easy to see that $J_k^k=I_k$.

\medskip

Let $M, N\in\mathcal{M}_k$. We say that $M$ and $N$ are equivalent, written as $M\sim N$, if there exist $(i_1,j_1), (i_2,j_2), \cdots, (i_l,j_l)\in\{1,2,\cdots,k\}\times\{1,2,\cdots,k\}$ such that
$$M=P(i_l,j_l)\cdots P(i_2,j_2)P(i_1,j_1)NP(i_1,j_1)P(i_2,j_2)\cdots P(i_l,j_l),$$ where $P(i,j)$ is the matrix obtained by interchanging the $i$-th and $j$-th rows of the unit matrix $I_k$.
By the definition, we have the implication
\begin{equation}\label{im1}
\mbox{$M\sim N\Rightarrow\|M^n\|=\|N^n\|$ for all $n\in\mathbb{N}$.}
\end{equation}
Note also that $I_2\not\sim J_2$. In fact, we have $P(1,2)J_2P(1,2)=J_2$ as is easily seen.

\begin{thm}\label{t3}
Let $M\in\mathcal{M}_b({\bf P_1})$. Then $\sup_{n\in\mathbb{N}}\|M^n\|=2^{b-1}$ if and only if
$M$ is equivalent with one of the following three matrices
\begin{equation}\label{m0}
\left(
      \begin{array}{cc}
        1 & \bf{0} \\
        \bf{1} & L_{b-1}  \\
      \end{array}
    \right),\
\left(
      \begin{array}{cc}
      I_2 & \bf{0} \\
      \bf{1} & L_{b-2} \\
    \end{array}
    \right),\
\left(
    \begin{array}{cc}
      J_2 & \bf{0} \\
      \bf{1} & L_{b-2} \\
    \end{array}
  \right).
\end{equation}
In the case $b=2$, these three matrices are $\left(
      \begin{array}{cc}
        1 & 0 \\
        1 & 0  \\
      \end{array}
    \right),\, I_2$, $J_2$, respectively. \end{thm}

\begin{thm}\label{t2}
Let $M\in\mathcal{M}_b({\bf P_1})$.
Then $\|M^n\|=C_{n+b}^{n+1}$ for all $n\in\mathbb{N}$ if and only if $M$ is equivalent with $T_b$.
\end{thm}

The above results can be directly applied to $M$-admissible words.
Let $b\in\mathbb{N},\, b\geq 2$. Let $A=\{1,2,\cdots, b\}$.
For every $n\in\mathbb{N}$ denote by $A^n$ all words of length $n$, by $A^*$ all finite words, and by $A^{\mathbb{N}}$ all infinite words over $A$. For a word $w$ denote by $w_k$ its $k$-th term and by $|w|$ its length. For $u\in A^*$ denote by $u^\infty$ the periodic infinite word of period $u$ and by $uw$ the concatenation.
In symbol dynamic system \cite{W}, $A^{\mathbb{N}}$ is an important metric space equipped with the metric
$$
d(u,v)=\left\{
         \begin{array}{ll}
           0 & \hbox{if $u=v$} \\
           b^{-k(u,v)} & \hbox{if $u\neq v$,}
          \end{array}
       \right.
$$ where $$k(u,v)=\min\{k\in\mathbb{N}: u_{k}\neq v_{k}\},\,\mbox{ if }u, v\in A^{\mathbb{N}},\, u\neq v.$$

Let $M\in\mathcal{M}_b$. Let $n\in \mathbb{N}$, $n\geq 2$. Let
$$
A^{n}_M:=\{w\in A^n:
M_{w_kw_{k+1}}=1 \mbox{ for all } 1\leq k<n\},
$$
$$A^*_M:=\bigcup_{n=2}^\infty A^{n}_M,$$ and
$$
A^{\mathbb{N}}_M:=\{w\in A^{\mathbb{N}}:
M_{w_kw_{k+1}}=1 \mbox{ for all } k\in\mathbb{N}\}.
$$
The words in these three sets are called $M$-admissible. Since
\begin{equation}\label{ys}
(M^{n})_{ij}=\sum_{w\in A^{n-1}}M_{iw_1}M_{w_1w_2}\cdots M_{w_{n-1}w_{n-1}} M_{w_{n-1}j}=\sharp A^{n+1}_M(i,j)
\end{equation}
for any $i, j\in A$, one has for each $n\in\mathbb{N}$
\begin{equation}\label{xdd}
\|M^{n}\|=\sharp A^{n+1}_M.
\end{equation}
Hereafter $(M^{n})_{ij}$ denotes the $(i,j)$-element of $M^{n}$, $A^{n}_M(i,j)$ denotes all $M$-admissible words of length $n$ of head $i$ and tail $j$, and $A^*_M(i,j)$ denotes all $M$-admissible finite words of head $i$ and tail $j$. Also, we know from \cite{BP}  that
\begin{equation}\label{dim}
\dim_HA^\mathbb{N}_M=\dim_BA^\mathbb{N}_M=\frac{\log\rho(M)}{\log b}
\end{equation}
where $\dim_H$ and $\dim_B$ denote the Hausdorff and box dimensions respectively. Using (\ref{xdd}), (\ref{dim}), Theorems \ref{t1} and \ref{t3}, we have the following results.

\begin{cor}\label{t4}
Let $M\in\mathcal{M}_b({\bf P_1})$. Then $A^{\mathbb{N}}_M$ obeys a trichotomy as follows.

(1) $\dim_HA^\mathbb{N}_M>0$.

(2) $A^\mathbb{N}_M$ is countably infinite.

(3) $\sharp A^{\mathbb{N}}_M\leq 2^{b-1}$. Hereafter $\sharp$ denotes the cardinality.

\end{cor}

\begin{cor}\label{t5}
Let $M\in\mathcal{M}_b({\bf P_1})$. Then $\sharp A^{\mathbb{N}}_M=2^{b-1}$ if and only if $M$ is equivalent with one of the three matrices in (\ref{m0}).
\end{cor}

\section{Preliminaries}

\subsection{Two known combinatoric formulas}

\begin{equation}\label{id1}
C_{1+b}^1+C_{1+b}^{2}+C_{2+b}^{3}+\cdots+C_{n-1+b}^{n}=C_{n+b}^{n},\,\mbox{ where } b, n\in\mathbb{N}.
\end{equation}
\begin{equation}\label{cd1}
kC_{n+b-k}^n+C_{n+b-k}^{n+1}\leq C_{n+b}^{n+1},\,\mbox{ where } k, b, n\in\mathbb{N},\, 1\leq k< b.
\end{equation}
The equality in (\ref{cd1}) holds if and only if $k=1$.

\subsection{Some basic facts on norms of matrices}

Let $k,s\in\mathbb{N}$. Let ${\bf 1}, {\bf 0}, I_k, L_k, T_k, J_k$ be matrices defined as in Section 1.

\begin{lemma}\label{2.3}
Let $b\in\mathbb{N},\, b\geq 2$. Then
\begin{equation}\label{slt}
b+\|L_{b-1}^1\|+\|L_{b-1}^2\|+\cdots+\|L_{b-1}^{b-2}\|=2^{b-1}.
\end{equation}
\end{lemma}
\begin{proof}
By induction we easily get for all $n\in\mathbb{N}$
$$
L_{b}^n=\left(
      \begin{array}{cc}
        0 & \bf{0} \\
        \bf{1} & L_{b-1}  \\
      \end{array}
    \right)^n
    =\left(
      \begin{array}{cc}
        0 & \bf{0} \\
        L_{b-1}^{n-1}\bf{1} & L_{b-1}^n  \\
      \end{array}
    \right),$$
which yields
\begin{equation}\label{slt1}
\|L_{b}^n\|=\|L_{b-1}^{n-1}\|+\|L_{b-1}^n\|.
\end{equation}

Now we are going to prove the equality (\ref{slt}). First it holds for $b=2$ obviously. Next assume that it holds for an arbitrarily given integer $b\geq 2$. Then it holds for $b+1$ because, by the equality (\ref{slt1}) and the inductive assumption, we have
\begin{eqnarray*}
  && b+1+\|L_{b}^1\|+\|L_{b}^2\|+\cdots+\|L_{b}^{b-1}\| \\
  &=& b+1+(b-1+\|L_{b-1}^1\|)+(\|L_{b-1}^1\|+\|L_{b-1}^2\|)+\cdots+(\|L_{b-1}^{b-2}\|+\|L_{b-1}^{b-1}\|) \\
  &=& 2(b+\|L_{b-1}^1\|+\|L_{b-1}^2\|+\cdots+\|L_{b-1}^{b-2}\|)=2^b.
\end{eqnarray*}
This completes the proof. \end{proof}

Denote by $\mathcal{U}_s$ the family of orthogonal matrices in $\mathcal{M}_s$, i.e.
$$\mathcal{U}_s=\{U\in\mathcal{M}_s: U^TU=I_s\}.$$

\begin{lemma}\label{U}
Let $U\in\mathcal{U}_s$.  Then we have the following statements.

1) $\|U\|=s$.

2) $\|VU\|=\|V\|$ for every nonnegative $k\times s$ matrix $V$.

3) Let $B$ be a nonnegative $k\times k$ matrix and $n\in\mathbb{N}$. Then we have
$$
(B{\bf 1})U=B{\bf 1}\,\mbox{ and }\,\,\left(
      \begin{array}{cc}
        U & \bf{0} \\
        \bf{1} & B  \\
      \end{array}
    \right)^{n}
    =\left(
      \begin{array}{cc}
        U^n & \bf{0} \\
        {\bf 1}+B{\bf 1}+\cdots+B^{n-1}{\bf 1}\, & B^n  \\
      \end{array}
    \right).$$

4) Let $b, n\in\mathbb{N}$ satisfy $0\leq b-s\leq n$. Then we have
$$\|\left(
      \begin{array}{cc}
        U & \bf{0} \\
        \bf{1} & L_{b-s}  \\
      \end{array}
    \right)^{n}\|=s2^{b-s}.$$
\end{lemma}
\begin{proof}
We only prove the statement 4). Let $b, n\in\mathbb{N}$ satisfy $0\leq b-s\leq n$. Using the above three statements, the equality $L_k^k=0$,  and Lemma \ref{2.3}, we have
\begin{eqnarray*}
 && \|\left(
      \begin{array}{cc}
        U & \bf{0} \\
        \bf{1} & L_{b-s}  \\
      \end{array}
    \right)^{n}\|  \\
   &=&  \|U^{n}\|+\|{\bf 1}+L_{b-s}^1{\bf 1}+L_{b-s}^2{\bf 1}+\cdots+L_{b-s}^{b-s-1}{\bf 1}\|\\
   &=&  \|U^{n}\|+\|{\bf 1}\|+\|L_{b-s}^1{\bf 1}\|+\|L_{b-s}^2{\bf 1}\|+\cdots+\|L_{b-s}^{b-s-1}{\bf 1}\|\\
   &=&  s+s(b-s+\|L_{b-s}^1\|+\|L_{b-s}^2\|+\cdots+\|L_{b-s}^{b-s-1}\|)= s2^{b-s}.
\end{eqnarray*}
This completes the proof. \end{proof}

\begin{lemma}\label{id35}
$\|T_b^n\|=C_{n+b}^{n+1}$ for all $n\in\mathbb{N}.$
\end{lemma}
\begin{proof} By a simple induction, one has $\|T_2^n\|=n+2=C_{n+2}^{n+1}$ for all $n\in\mathbb{N}$. Thus the desired equality is true for $b=2$. Suppose that it is true for an arbitrarily given integer $b\geq 2$.
Then, using Lemma \ref{U} and (\ref{id1}), we get
\begin{eqnarray*}
\|T_{b+1}^n\| &=& \|\left(
      \begin{array}{cc}
        1 & \bf{0} \\
        \bf{1} & T_b  \\
      \end{array}
    \right)^n\| \\
   &=& 1+b+\|T_{b}^1\|+\cdots+\|T_{b}^{n-1}\|+\|T_{b}^n\| \\
   &=& C_{1+b}^1+C_{1+b}^2+\cdots+C_{n-1+b}^n+C_{n+b}^{n+1}=C_{n+b+1}^{n+1}.
\end{eqnarray*}
Thus the desired equality is true for $b+1$. This proves this lemma.
\end{proof}

\subsection{Some basic facts on $M$-admissible words}

\begin{lemma}\label{l1}
Let $M,N\in\mathcal{M}_b$. Then $M$ is equivalent with $N$ if and only if there is a permutation $\sigma:A\to A$ such that for every integer $n\geq 2$
\begin{equation}\label{comu}
A_M^{n}=\{\sigma(w): w\in A_N^{n}\},
\end{equation}
where we prescribe $\sigma(w)=\sigma(w_1)\sigma(w_2)\cdots\sigma(w_n)$.
\end{lemma}

\begin{proof} Let $i,j\in A$. Denote by $\sigma_{ij}$ the permutation of $A$ defined by $$\mbox{$\sigma_{ij}(i)=j$, $\sigma_{ij}(j)=i$, and $\sigma_{ij}(k)=k$ for all $k\in A\setminus\{i,j\}$.}$$ Note that $P(i, j)^2=I_b$. We easily see that $N=P(i,j)MP(i,j)$ if and only if
\begin{equation}\label{ey}
(N^{n})_{lk}=(M^{n})_{\sigma_{ij}(l)\sigma_{ij}(k)}
\end{equation}
for all $l,k\in A$ and all $n\in\mathbb{N}$, which in turn is obviously equivalent with $$A_M^{n}=\{\sigma_{ij}(w): w\in A_N^{n}\}$$ for all $n\geq 2$. It then follows that
$$N=P(i_l,j_l)\cdots P(i_2,j_2)P(i_1,j_1)MP(i_1,j_1)P(i_2,j_2)\cdots P(i_l,j_l)$$
if and only if the equality (\ref{comu}) holds  with $\sigma=\sigma_{i_{l}j_{l}}\circ\cdots\circ\sigma_{i_{2}j_{2}}\circ\sigma_{i_1j_1}$.
\end{proof}

Let $M\in\mathcal{M}_b$. Let
\begin{equation}\label{d0}
D_M=\{i\in A: (M^k)_{ii}\geq 1\mbox{ for some }k\in\mathbb{N}\}.
\end{equation}
Then by (\ref{ey}) we have the implication
\begin{equation}\label{im2}
N\sim M\Rightarrow\sharp D_N=\sharp D_M.
\end{equation}
Suppose $M$ is strictly lower triangular, i.e. $M$ satisfies $M_{ij}=0$ for all $i,j\in A$ with $i\leq j$. Then $D_M=\emptyset$. Furthermore, we have the following lemma.

\begin{lemma}\label{empty}
Let $M\in\mathcal{M}_b$. The following three conditions are equivalent.

1) $D_M=\emptyset$.

2) $M^{b}=\bf{0}$.

3) $M$ is equivalent with a strictly lower triangle matrix.
\end{lemma}
\begin{proof}
$1)\Rightarrow 2)$. Suppose $M^{b}\neq\bf{0}$. Then $A^{b+1}_M\neq\emptyset$. Let $j_1j_2\cdots j_{b+1}\in A^{b+1}_M$. By the pigeonhole principle, there is some $i\in A$ and $1\leq t< s\leq b+1$ such that $i=j_t=j_s$, so
$A^{s-t+1}_M(i,i)\neq\emptyset$ or equivalently $(M^{s-t})_{ii}\geq 1$, and so $D_M\neq\emptyset$.

$2)\Rightarrow 1)$. If $D_M\neq\emptyset$, there is some $i\in A$ and $k\in\mathbb{N}$ such that $(M^k)_{ii}\geq 1$, so $(M^{kn})_{ii}\geq 1$ for all $n\in\mathbb{N}$, which implies $M^{b}\neq\bf{0}$.

$3)\Rightarrow 1)$. This is obvious.

$1)\Rightarrow 3)$. Suppose $D_M=\emptyset$. We are going to show that $M$ is equivalent with a matrix $N$ with the property: $j_1>j_2>\cdots>j_l$ for each $j_1j_2\cdots j_l\in A^{*}_N$. This property implies that $N$ is strictly lower triangular.

As was shown, we have $M^b=\bf{0}$, so $A^{b+1}_M=\emptyset$ by (\ref{xdd}), which implies that $A^{*}_M$ has the property:
If $w_1w_2\cdots w_n\in A^{*}_M$, then $n\leq b$, the letters $w_1,w_2,\cdots, w_n$ are pairwise distinct, and $A^{*}_M$ contains no words of head $w_t$ and tail  $w_s$ for any $1\leq s<t\leq n$.

Without loss of generality assume $A^{*}_M\neq \emptyset$. Let $w=w_1w_2\cdots w_n\in A^{*}_M$ be given. By the above property of $A_M^*$ there is a permutation $\sigma:A\to A$ such that
\begin{equation}\label{order}
\sigma(w_1)>\sigma(w_2)>\cdots>\sigma(w_n).
\end{equation}
Then, by Lemma \ref{l1}, $M$ is equivalent with a matrix $B$ with $\sigma(w)\in A_B^*$. By (\ref{im1}) and (\ref{xdd}) one has $\sharp A_B^*=\sharp A_M^*$. By (\ref{im2}) one has $D_B=\emptyset$, so $A_B^*$ shares the above property of $A_M^*$.

Now let $u=u_1u_2\cdots u_l\in A^{*}_B\setminus\{\sigma(w)\}$, if it exists. Then $u_s>u_t$ by (\ref{order}), provided that $u_s,u_t\in\{\sigma(w_1),\sigma(w_2),\cdots, \sigma(w_n)\}$ and $1\leq s<t\leq l$. Therefore there is a permutation $\tau:A\to A$ such that $\tau(\sigma(w))=\sigma(w)$ and $$\tau(u_1)>\tau(u_2)>\cdots>\tau(u_l).$$
Using again Lemma \ref{l1}, $M$ is equivalent with a matrix $C$ with $\tau(u), \sigma(w)\in A_C^*$. Continuing this argument, we see that $M$ is equivalent with a matrix $N$ with the desired property. This completes the proof.
\end{proof}

\begin{lemma}\label{nP}
Let $M\in\mathcal{M}_b({\bf P_1})$. Then $D_M\neq\emptyset$. Moreover, if $(M^k)_{ii}>1$ for some $k\in \mathbb{N}$ and some $i\in A$ then $\{\|M^n\|\}_{n=1}^\infty$ tends to $\infty$ exponentially.
\end{lemma}
\begin{proof}
We remark that $M\in\mathcal{M}_b({\bf P_1})$ if and only if $\{\sharp A_M^n\}_{n=2}^\infty$, or equivalently, $\{\|M^n\|\}_{n=1}^\infty$ is a nondecreasing sequence of positive integers. In particular, we have $M^b\neq{\bf 0}$, so $D_M\neq\emptyset$ by Lemma \ref{empty}.

Suppose $(M^k)_{ii}>1$ for some $k\in \mathbb{N}$ and some $i\in A$. Then we have $$\|M^{kn}\|\geq((M^k)_{ii})^n\geq2^n$$ for all $n\in\mathbb{N}$. Therefore $\{\|M^n\|\}_{n=1}^\infty$ tends to $\infty$ exponentially.
\end{proof}

Let $M\in\mathcal{M}_b$. We say that $M$ satisfies the condition ${\bf P_2}$, if $(M^k)_{ii}\leq1$ for all $i\in A$ and all $k\in \mathbb{N}$. Let
$$\mathcal{M}_b({\bf P_1, P_2})=\{M\in\mathcal{M}_b: M \mbox{ satisfies }{\bf P_1}\mbox{ and }{\bf P_2}\}.$$

\begin{lemma}\label{if}
Let $M\in\mathcal{M}_b({\bf P_1, P_2})$. Then we have the following statements.

1) $A_M^{\mathbb{N}}$ has a unique periodic word of head $i$ for each $i\in D_M$.

2) $A_M^{\mathbb{N}}$ has no periodic word of head $j$ for any $j\in A\setminus D_M$.

3) Every word of $A_M^{\mathbb{N}}$ is ultimately periodic.
\end{lemma}
\begin{proof}
1) Let $M\in\mathcal{M}_b({\bf P_1, P_2})$. We have by (\ref{d0}) and Lemma \ref{nP}
$$
D_M=\{i\in A: (M^k)_{ii}=1\mbox{ for some }k\in\mathbb{N}\}\neq\emptyset.
$$
Let $i\in D_M$ and let $k\in\mathbb{N}$ be the smallest integer such that $(M^k)_{ii}=1$. By (\ref{ys}), $A_M^{k+1}(i,i)$ has a unique word, denoted as $ii_2i_3\cdots i_ki$. Then $i, i_2, \cdots, i_k$ are pairwise distinct and belong to $D_M$.
Let $w=ii_2i_3\cdots i_k$. Then $w^\infty$ is a periodic word of $A_M^{\mathbb{N}}$ of head $i$. To prove it is unique, suppose $u^\infty\in A^{\mathbb{N}}_M$, where $u=iu_2u_3\cdots u_l\neq w$, and $i, u_2, \cdots, u_l$ are pairwise distinct. Then $u^{k+l}i$ and  $(uw)^{l}i$ are two distinct words in $A_M^{l(k+l)+1}(i,i)$, giving $(M^{l(k+l)})_{ii}\geq 2$, contradicting the condition ${\bf P_2}$.

2) For each $j\in A\setminus D_M$ one has $(M^k)_{jj}=0$ for all $k\in\mathbb{N}$ by (\ref{d0}) and ${\bf P_2}$, so $A_M^{k}(j,j)=\emptyset$ for all $k\geq 2$, and so $A_M^{\mathbb{N}}$ has no periodic word of head $j$.

3) Let $u\in A_M^{\mathbb{N}}$. There is a digit $i\in A$ and a strictly increasing sequence $\{n_k\}_{k=1}^\infty$ of positive integers such that
$u_n=i$ if and only if $n\in\{n_k: k\in\mathbb{N}\}$, so $(iu_{n_k+1}\cdots u_{n_{k+1}-1})^\infty$ is a periodic word of head $i$ in $A_M^{\mathbb{N}}$ for each $k\in\mathbb{N}$. Now by 1) and 2) we have $i\in D_M$ and
$$iu_{n_k+1}\cdots u_{n_{k+1}-1}=iu_{n_1+1}\cdots u_{n_{2}-1}$$
for all $k\in\mathbb{N}$. Thus $u$ is ultimately periodic.
\end{proof}

\begin{lemma}\label{iif}
Let $M\in\mathcal{M}_b({\bf P_1, P_2})$. The following conditions are equivalent.

1) $\{\|M^n\|\}_{n=1}^\infty$ is bounded.

2) $A_M^{\mathbb{N}}$ has a unique word of head $i$ for each $i\in D_M$.
\end{lemma}

\begin{proof}
$1)\Rightarrow 2).$ Suppose $\{\|M^n\|\}_{n=1}^\infty$ is bounded. Then $\{\sharp A_M^{n}\}_{n=2}^\infty$ is bounded, so there is an integer $n_0\geq 2$ such that  $\sharp A_M^{n} =\sharp A_M^{n_0}$ for all $n\geq n_0$. Therefore $A^{\mathbb{N}}_M$ is a finite set.  Let $i\in D_M$. By Lemma \ref{if}, $A_M^{\mathbb{N}}$ has a periodic word $w^\infty$ of head $i$. We conclude that it is the only word of $A_M^{\mathbb{N}}$ of head $i$. Indeed, if $A^{\mathbb{N}}_M$ has a word $u$ of head $i$ with $u\neq w^\infty$ then $\{w^ku: k\in \mathbb{N}\}\subset A^{\mathbb{N}}_M$, which implies that $A^{\mathbb{N}}_M$ is infinite, a contradiction.

$2)\Rightarrow 1).$  Suppose, for each $i\in D_M$, $A_M^{\mathbb{N}}$ has a unique word of head $i$, which is denoted by $w^{(i)}$. Then these words are periodic by Lemma \ref{if}.

Let
$$U=\{u\in A_M^*:  u_1,u_2,\cdots,u_{|u|}\in A\setminus D_M\}.$$
Since $A_M^{\mathbb{N}}$ has no periodic word of head $j\in A\setminus D_M$, by the pigeonhole principle we have $|u|\leq \sharp(A\setminus D_M)$ for each $u\in U$. Thus $U$ is a finite set. Observing that $$A_M^{\mathbb{N}}\subseteq\{w^{(i)}: i\in D_M\}\cup\{uw^{(i)}: u\in U, i\in D_M\},$$ we see that $A^{\mathbb{N}}_M$ is a finite set, which implies that $\{\|M^n\|\}_{n=1}^\infty$ is bounded.
\end{proof}

Let $M\in\mathcal{M}_b({\bf P_1, P_2})$. For each $i\in D_M$ denote by $D_{M,i}$ the set consisting of all letters of the periodic word of head $i$ in $A_M^{\mathbb{N}}$.
Then $j\in D_{M,i}$ if and only if $i\in D_{M,j}$ for $i,j\in D_M$. Let
\begin{equation}\label{ok}
D_{M}^0=\{i\in D_M: A_M^*(j,l)=\emptyset \mbox{ for all } j\in D_{M,i} \mbox{ and all } l\in A\setminus D_M\}.\end{equation}
Then for each $i\in D_M^0$ we have

(a) $D_{M,i}\subset D_M^0$, and

(b) $M_{jl}=0$ for all $j\in D_{M,i}$ and all $l\in A\setminus D_M$.
\begin{lemma}\label{ijf}
Let $M\in\mathcal{M}_b({\bf P_1, P_2})$. Then $D_{M}^0\neq\emptyset$.
\end{lemma}
\begin{proof}
Suppose $D_{M}^0=\emptyset$. Then, given $i_1\in D_M$, there is a $j_1\in D_{M,i_1}$ and an $l_1\in A\setminus D_M$ such that $A_M^*(j_1,l_1)\neq\emptyset$. By Lemma \ref{if}, $A_M^*(l_1,i_2)\neq\emptyset$ for some  $i_2\in D_M$. Repeating this argument, there is a $j_2\in D_{M,i_2}$ and an $l_2\in A\setminus D_M$ such that $A_M^*(j_2,l_2)\neq\emptyset$ and $A_M^*(l_2,i_3)\neq\emptyset$ for some $i_3\in D_M$. Continuing  this argument, we see that $A_M^{\mathbb{N}}$ has a word of the form
$$i_1\cdots j_1\cdots l_1\cdots i_2\cdots j_2\cdots l_2\cdots i_3\cdots j_3\cdots l_3\cdots,$$
where $i_k\in D_M$, $j_k\in D_{M, i_k}$ and $l_k\in A\setminus D_M$ for all $k\geq 1$. By the pigeonhole principle, there are $1\leq s<t$ such that $l_s=l_t$, which implies that $A_M^{\mathbb{N}}$ has a periodic word of head $l_s$, a contradiction.
\end{proof}
Let $$D_M^{00}=\{i\in D_M^0: M_{ju}=0 \mbox{ for all } j\in  D_{M,i} \mbox{ and all } u\in D_M\setminus D_{M,i}\}.$$

\begin{lemma}\label{iijf}
Let $M\in\mathcal{M}_b({\bf P_1, P_2})$. Then $D_M^{00}\neq\emptyset$.
\end{lemma}
\begin{proof}  As mentioned, we have $D_{M,i}\subset D_M^0$ for each $i\in D_M^0$. Thus
there exist $s_1,s_2,\cdots, s_k\in D_M^0$ such that
$D_{M,s_1},D_{M,s_2},\cdots, D_{M, s_k}$ are pairwise disjoint with
$$D_M^0=\cup_{t=1}^kD_{M,s_t}.$$ By Lemma \ref{if}, we may assume that $s_1,s_2,\cdots, s_k$ have be arranged such that
\begin{equation}\label{ook}
\mbox{$A_M^*(j,j')=\emptyset$ for all $j\in D_{M,s_t}$ and $j'\in D_{M, s_{t'}}$, if $1\leq t< t'\leq k$.}
\end{equation}
We are going to show $s_1\in D_M^{00}$.

Suppose that $s_1\not\in D_M^{00}$. This assumption together with (\ref{ook}) implies that $M_{j_1u_1}\neq 0$ for some $j_1\in  D_{M,s_1}$ and some $u_1\in D_M\setminus D_M^0$.  Then $A_M^*(s_1,l_1)\neq\emptyset$ for some $l_1\in A\setminus D_M$ by the definition of $D_M^0$, which together with (\ref{ook}) and Lemma \ref{if} implies that $A_M^*(l_1,u_2)\neq\emptyset$ for some $u_2\in D_M\setminus D_{M}^0$. Repeatedly arguing as above, we get an infinite word $w\in A^{\mathbb{N}}_M$, in which there are infinitely many letters from $A\setminus D_M$. This implies by the pigeonhole principle that $A_M^{\mathbb{N}}$ has a periodic word of head $l$ for some $l\in A\setminus D_M$, a contradiction.
\end{proof}

\section{The proof of main results}

Let $b\in\mathbb{N},\, b\geq 2$. Let $A=\{1,2,\cdots, b\}$. Let $M\in\mathcal{M}_b$. For each $i\in A$ let
$$A_M^{\mathbb{N}}(i)=\{w\in A_M^{\mathbb{N}}: w_1=i\}.$$

\begin{lemma}\label{boun}
Let $M\in\mathcal{M}_b({\bf P_1})$. Suppose that $\{\|M^n\|\}_{n=1}^\infty$ is bounded. Then $\|M^n\|\leq 2^{b-1}$ for all $n\in\mathbb{N}$.
\end{lemma}
\begin{proof} The conditions here imply $M\in\mathcal{M}_b({\bf P_1,P_2})$ by Lemma \ref{nP}. Therefore $\sharp A_M^{\mathbb{N}}(i)=1$ for each $i\in D_M$ by Lemma \ref{iif}. It then follows from Lemma \ref{l1} and Lemma \ref{empty} that $M$ is equivalent with a matrix $N$ of the form
$$
\left(
  \begin{array}{cc}
    U & {\bf 0}  \\
    V & B \\
  \end{array}
\right),
$$
where $U\in\mathcal{U}_s$; $s=\sharp D_M\geq 1$; $B$ is a strictly lower triangular matrix. Now using Lemma \ref{U}, we get
\begin{equation}\label{iff}
\|M^n\|=\|N^n\|\leq\|\left(
  \begin{array}{cc}
    U & 0  \\
    {\bf 1} & L_{b-s} \\
  \end{array}
\right)^{n}\|\leq s2^{b-s}\leq 2^{b-1}
\end{equation}
for all $n\in \mathbb{N}$.
This completes the proof.
\end{proof}

\begin{lemma}\label{unb}
Let $M\in\mathcal{M}_b({\bf P_1,P_2})$. Then $\|M^n\|\leq C_{n+b}^{n+1}$ for all $n\in \mathbb{N}$.
\end{lemma}
\begin{proof}
The lemma is true for $b=2$ by a simple computation. Suppose it is true for all $2\leq l<b$, where $b>2$ is an arbitrarily given integer. We are going to show that it is true for $b$.

Let $M\in\mathcal{M}_b({\bf P_1,P_2})$. By passing to an equivalent matrix by Lemma \ref{l1}, Lemma \ref{if} and Lemma \ref{iijf}, we may assume without loss of generality that $1\in D_M^{00}$ and that $(12\cdots k)^\infty$ is the unique periodic word of $A_M^{\mathbb{N}}(1)$ for some $k\leq b$.
Then by Lemma \ref{ijf} and Lemma \ref{iijf} we see that $M$ is of the form
$$\left(
  \begin{array}{cc}
    U & \textbf{0} \\
    V & B \\
  \end{array}
\right),
$$
where $$U=\left\{
         \begin{array}{ll}
           I_1 & \hbox{if $k=1$} \\ \\
           J_k & \hbox{otherwise}.
          \end{array}
       \right.$$
Let
$$
N=\left(
  \begin{array}{cc}
    U & {\bf 0} \\
    {\bf 1} & B \\
  \end{array}
\right).
$$
Then we easily get
\begin{equation}\label{mnbj}
\sharp A_M^{n+1}\leq \sharp A_N^{n+1}\mbox{ or equivalently }\|M^n\|\leq\|N^n\| \mbox{ for all $n\in\mathbb{N}$.}
\end{equation}
Clearly, $B$ satisfies the condition ${\bf P_2}$. But it is possible that $B$ does not satisfies the condition ${\bf P_1}$. Let $\widehat{B}$ be defined by
$$
r_i(\widehat{B})=\left\{
         \begin{array}{ll}
           e_i & \hbox{if $r_i(B)=0$} \\ \\
           r_i(B) & \hbox{otherwise}
          \end{array}
       \right.
$$
for all $i\in \{1,2,\cdots, b-k\}$, where $e_i$ denotes the $(b-k)$-dimensional unit row vector whose $i$-th component is $1$. Then
$\widehat{B}\in\mathcal{M}_{b-k}({\bf P_1, P_2})$ and
\begin{equation}\label{kk}
\|B^n\|\leq \|\widehat{B}^n\|
\end{equation}
for all $n\in\mathbb{N}$.
By the inductive assumption, $\|\widehat{B}^n\|\leq C_{n+b-k}^{n+1}$ for all $n\in\mathbb{N}$.
It then follows  from Lemma \ref{U}, (\ref{id1}) and (\ref{cd1}) that for every $n\in\mathbb{N}$
\begin{eqnarray*}
&&\|M^n\|\leq\|N^n\|\\&=&\|U^n\|+\|{\bf 1}\|+\|B{\bf 1}\|+\cdots+\|B^{n-1}{\bf 1}\|+\|B^n\| \\
   &\leq& k(1+b-k +\|\widehat{B}\|+\cdots+\|\widehat{B}^{n-1}\|)+\|\widehat{B}^n\|\\
   &\leq& k(C_{1+b-k}^1+C_{1+b-k}^2+C_{2+b-k}^3+\cdots+C_{n-1+b-k}^n)+C_{n+b-k}^{n+1}\\
   &\leq& kC_{n+b-k}^n+C_{n+b-k}^{n+1}\leq C_{n+b}^{n+1}.
\end{eqnarray*}
This completes the proof.
\end{proof}

Now we are ready to prove Theorems \ref{t1}, \ref{t3} and \ref{t2}.

\medskip

\noindent{\it Proof of Theorem \ref{t1}.} Let $b\in \mathbb{N},\, b\geq 2$. Let $M\in\mathcal{M}_b({\bf P_1})$. We are required to prove the trichotomy:

1) $\{\|M^n\|\}_{n=1}^\infty$ tends to $\infty$ exponentially.

2)  $n+2\leq\|M^n\|\leq C_{n+b}^{n+1}$ for all $n\in\mathbb{N}$.

3) $\|M^n\|\leq 2^{b-1}$ for all $n\in\mathbb{N}$.

\noindent Since Lemmas \ref{nP}, \ref{boun} and \ref{unb} have been proved, it suffices to show that if $M\in\mathcal{M}_b({\bf P_1, P_2})$ and  $\{\|M^n\|\}_{n=1}^\infty$ is unbounded, then $\|M^n\|=\sharp A_M^{n+1}\geq n+2$ for all $n\in\mathbb{N}$. In fact, by Lemmas \ref{if} and \ref{iif}, there is an $i\in D_M$ such that $A_M^{\mathbb{N}}(i)$ has at least two words, by which we easily show $\sharp A_M^{n+1}\geq n+2$ for all $n\in\mathbb{N}$ by induction. This completes the proof.

\medskip

\noindent{\it Proof of Theorem \ref{t3}.} Let $b\in \mathbb{N},\, b\geq 2$. Let $M\in\mathcal{M}_b({\bf P_1})$. We are required to prove that the following two statements are equivalent.

1) $\sup_{n\in\mathbb{N}}\|M^n\|=2^{b-1}$.

2) $M$ is equivalent with one of the following three matrices
$$
\left(
      \begin{array}{cc}
        1 & \bf{0} \\
        \bf{1} & L_{b-1}  \\
      \end{array}
    \right),\
\left(
      \begin{array}{cc}
      I_2 & \bf{0} \\
      \bf{1} & L_{b-2} \\
    \end{array}
    \right),\
\left(
    \begin{array}{cc}
      J_2 & \bf{0} \\
      \bf{1} & L_{b-2} \\
    \end{array}
  \right).
$$

2) $\Rightarrow$ 1). If the statement 2) is supposed, we get $\sup_{n\in\mathbb{N}}\|M^n\|=2^{b-1}$ by (\ref{im1}) and Lemma \ref{U} immediately.

1) $\Rightarrow$ 2). Suppose $\sup_{n\in\mathbb{N}}\|M^n\|=2^{b-1}$. Then we have $M\in\mathcal{M}_b({\bf P_1, P_2})$ by Lemma \ref{nP}. Next, by the proof of Lemma \ref{boun}, $M$ is equivalent with a matrix of the form $$\left(
  \begin{array}{cc}
    U & {\bf 0}  \\
    V & B \\
  \end{array}
\right),$$
where $U\in\mathcal{U}_s$ for some $1\leq s\leq b$, $B$ is a strictly lower triangular matrix, and  $$
\|\left(
  \begin{array}{cc}
    U & {\bf 0}  \\
    V & B \\
  \end{array}
\right)^n\|=\|\left(
  \begin{array}{cc}
    U & 0  \\
    {\bf 1} & L_{b-s} \\
  \end{array}
\right)^{n}\|= s2^{b-s}= 2^{b-1}
$$
for sufficiently large $n\in \mathbb{N}$. Thus $V={\bf 1}$, $B=L_{b-s}$, and $s=1$ or $2$. Then one has $U=I_1$ if $s=1$, and $U=I_2$ or $J_2$ if $s=2$. This completes the proof.

\medskip

\noindent{\it Proof of Theorem \ref{t2}.}
Let $b\in \mathbb{N},\, b\geq 2$. Let $M\in\mathcal{M}_b({\bf P_1})$. We are required to prove that
the following two statements are equivalent.

1) $\|M^n\|=C_{n+b}^{n+1}$ for all $n\in\mathbb{N}$.

2) $M\sim T_b$.

2) $\Rightarrow$ 1). If $M\sim T_b$ is supposed, we get $\|M^n\|=C_{n+b}^{n+1}$ for all $n\in\mathbb{N}$ by (\ref{im1}) and Lemma \ref{id35} immediately.

1) $\Rightarrow$ 2).  Suppose that $\|M^n\|=C_{n+b}^{n+1}$ for all $n\in\mathbb{N}$. Then $M\in\mathcal{M}_b({\bf P_1, P_2})$ by Lemma \ref{nP}. Next, by the proof of Lemma \ref{unb}, $M$ is equivalent with a matrix of the form $$
\left(
  \begin{array}{cc}
    U & {\bf 0} \\
    V & B \\
  \end{array}
\right),
$$
where $U\in\mathcal{U}_k$ for some $1\leq k\leq b$ and
\begin{eqnarray*}
  && \|\left(
  \begin{array}{cc}
    U & {\bf 0} \\
    V & B \\
  \end{array}
\right)^n\|=\|\left(
  \begin{array}{cc}
    U & {\bf 0} \\
    {\bf 1} & B \\
  \end{array}
\right)^n\| \\
 &=& kC_{n+b-k}^n+\|\widehat{B}^n\|=kC_{n+b-k}^n+C_{n+b-k}^{n+1}=C_{n+b}^{n+1}
\end{eqnarray*}
for all $n\in\mathbb{N}$. Thus $k=1$, $V={\bf 1}$, $B=\widehat{B}\in\mathcal{M}_{b-1}({\bf P_1, P_2})$, and $\|B^n\|=C_{n+b-1}^{n+1}$ for all $n\in\mathbb{N}$. This proves that $M$ is equivalent with a matrix of the form
$$\left(
  \begin{array}{cc}
    1 & {\bf 0} \\
    {\bf 1} & B \\
  \end{array}
\right).$$
 If $b>2$,  doing the same thing as above for the matrix $B$, we see that $M$ is equivalent with a matrix of the form
$$\left(
  \begin{array}{ccc}
    1 & 0 & {\bf 0} \\
    1 & 1 & {\bf 0} \\
    {\bf 1} & {\bf 1} & X \\
  \end{array}
\right),
$$
where $X$ satisfies $X\in\mathcal{M}_{b-2}({\bf P_1, P_2})$ and $\|X^n\|=C_{n+b-2}^{n+1}$ for all $n\in\mathbb{N}$. Now we get $M\sim T_b$ by induction.


\begin{thebibliography}{}
\bibitem{BP} C. J. Bishop and Y. Peres, Fractal sets in Probability and Analysis, Cambridge University Press, 2017.

\bibitem{BH} R. A. Brualdi and S. G. Hwang, On the Spectral Radius of $(0,1)$-Matrices with $1$'s in Prescribed Positions,
SIAM Journal on Matrix Analysis and Applications, 1996, Vol.17, No.3: pp. 489-508.

\bibitem{S02} D. Serre, Matrics Theory and Applications, Springer, 2002.

\bibitem{W} P. Walters, An Introduction to Ergodic Theory, Springer, New York, 2000.




\end{thebibliography}
\end{document}